\documentclass[11pt,twoside]{amsart}
\usepackage{amsxtra}
\usepackage{color}
\usepackage{amsmath,amsthm,amssymb,bm}
\usepackage{mathrsfs,mathtools}
\usepackage{stmaryrd}
\usepackage{hyperref}
\usepackage{enumerate}
\usepackage{xcolor}
\usepackage{pifont}
\usepackage{adjustbox}
\usepackage{tikz}

\theoremstyle{theorem}
\newtheorem{theorem}{Theorem}[section]
\newtheorem{corollary}[theorem]{Corollary}
\newtheorem{proposition}[theorem]{Proposition}
\newtheorem{lemma}[theorem]{Lemma}
\newtheorem*{mtheorem}{Main Theorem}
\newtheorem*{corollary*}{Corollary}
\newtheorem*{proposition*}{Proposition}
\theoremstyle{definition}

\theoremstyle{remark}

\newcommand{\R}{{\mathbb R}}

\newcommand{\T}{{\mathbb T}}

\newcommand{\beq}{\begin{equation}}
\newcommand{\eeq}{\end{equation}}

\newcommand{\s}{\sigma}

\newcommand{\SU}{{\mathrm{SU}}}
\newcommand{\On}{{\mathrm O}}

\newcommand{\M}{{\mathrm M}}
\renewcommand{\T}{{\mathrm T}}

\newcommand{\G}{{\mathrm G}}
\newcommand{\K}{{\mathrm K}}
\renewcommand{\L}{{\mathrm L}}

\newcommand{\Ng}{{\mathrm N}}

\DeclareMathOperator\Ad{Ad}

\newcommand{\Ric}{{\rm Ric}}

\newcommand{\fre}{\mathfrak{e}}

\newcommand{\gc}{\mathfrak{c}}

\newcommand{\gf}{\mathfrak{f}}
\renewcommand{\gg}{\mathfrak{g}}

\newcommand{\gk}{\mathfrak{k}}
\newcommand{\gl}{\mathfrak{l}}
\newcommand{\gm}{\mathfrak{m}}

\newcommand{\gp}{\mathfrak{p}}
\newcommand{\gq}{\mathfrak{q}}

\newcommand{\gu}{\mathfrak{u}}

\newcommand{\so}{\mathfrak{so}}
\newcommand{\su}{\mathfrak{su}}

\newcommand{\gsp}{\mathfrak{sp}}

\newcommand{\st}{\ |\ }

\newcommand{\onabla}{\overline{\nabla}}
\newcommand{\og}{\overline{g}}
\newcommand{\oH}{\overline{H}}

\numberwithin{equation}{section}

\title[Infinite families of homogeneous  Bismut Ricci flat manifolds]{Infinite families of homogeneous  Bismut Ricci flat manifolds}
\author{Fabio Podest\`a and Alberto Raffero}
\address{Dipartimento di Matematica e Informatica ``U.~Dini'' \\ Universit\`a degli Studi di Firenze\\ Viale Morgagni 67/a\\ 50134 Firenze\\ Italy\\ }
\email{fabio.podesta@unifi.it}
\address{Dipartimento di Matematica ``G.~Peano'' \\ Universit\`a degli Studi di Torino\\
Via Carlo Alberto 10\\
10123 Torino\\ Italy}
\email{alberto.raffero@unito.it}

\begin{document}

\begin{abstract} 
Starting from compact symmetric spaces of inner type, we provide infinite families of compact homogeneous spaces carrying invariant non-flat Bismut connections with vanishing Ricci tensor.
These examples turn out to be generalized symmetric spaces of order $4$ and (up to coverings) can be realized as minimal submanifolds of the Bismut flat model spaces, namely compact Lie groups. 
This construction generalizes the standard Cartan embedding of symmetric spaces. 
\end{abstract}

\maketitle

\section{Introduction}\label{intro}
\smallskip

On a Riemannian $n$-manifold $(\M,g)$, metric connections $\nabla$ are completely characterized by their torsion tensor 
\[
T_XY = \nabla_XY-\nabla_YX-[X,Y], \quad X,Y\in\Gamma(T\M),
\] 
and they can be classified into eight families determined by the irreducible $\On(n)$-module decomposition of the space $\R^n \otimes \Lambda^2{\R^n}$ \cite{Car,TV}. 
The class of metric connections with totally skew-symmetric torsion is made up of those connections for which the 3-covariant tensor 
\[
H(X,Y,Z) \coloneqq g(T_XY,Z)
\] 
is totally skew-symmetric. Every such connection is related to the Levi Civita connection $\nabla^g$ of $(\M,g)$ as follows 
\begin{equation}\label{BismutConn}
g(\nabla_XY,Z) = g(\nabla^g_XY,Z) + \frac12 H(X,Y,Z),
\end{equation}
and it has the same geodesics as $\nabla^g$.   

Consistently with \cite{FS}, given a 3-form $H\in\Omega^3(\M)$, we shall refer to the metric connection $\nabla$ defined  in \eqref{BismutConn} as the {\em Bismut connection} induced by the pair $(g,H)$. 

\smallskip

In complex non-K\"ahler geometry, a typical example is given by the Bismut connection of a Hermitian manifold $(\M,g,J)$, 
which is the unique metric connection with totally skew-symmetric torsion preserving the complex structure $J$  \cite{Bis,Gau}. 
In such a case, the torsion is given by $H = d^c\omega$, where $\omega=g(J\cdot,\cdot)$ is the fundamental $2$-form.  
If the torsion is also closed, so that $dd^c\omega=0$, one obtains the widely studied class of strong K\"ahler with torsion complex manifolds (SKT).  
On the other hand, Riemannian manifolds carrying a Bismut connection $\nabla$ with $\nabla$-parallel torsion have gained a lot of attention in the literature, 
since they are naturally associated with various geometrically meaningful structures as naturally reductive spaces, nearly K\"ahler and Sasakian structures as well as nearly parallel G$_2$-structures among others, 
see e.g.~\cite{Agr,AFF,CMS,FI}.

Metric connections with totally skew symmetric torsion are also an important object of interest in theoretical and mathematical physics, e.g.~in Type II string theory or in supergravity theories,
see \cite{Agr,FI, IP} for more details and references.

\smallskip

In this paper, we will focus on Bismut connections $\nabla$ with closed torsion form $H$ and zero Ricci tensor $\Ric^\nabla$.  
In generalized Riemannian geometry, Bismut connections with closed torsion play a prominent role, as they are naturally associated with generalized metrics on exact Courant algebroids.  
Moreover, the vanishing of $\Ric^\nabla$ implies that the associated generalized metric  is generalized Einstein, see \cite{Gar,FS}.

Since the torsion of $\nabla$ is non-vanishing, the Ricci tensor $\Ric^{\nabla}$ is not symmetric, and one has (see \cite[Prop.~3.18]{FS})
\[
\Ric^\nabla = \Ric_g -\frac14 H^2 - \frac12 \delta_g H,
\]
where $\Ric_g$ denotes the Ricci tensor of $\nabla^g$, $\delta_g$ is the formal adjoint of $d$, and the symmetric 2-tensor $H^2$ is defined as $H^2(X,Y) \coloneqq g(\imath_XH,\imath_YH)$, for every $X,Y\in\Gamma(T\M)$.

Therefore, a Bismut connection $\nabla$ with closed torsion form $H$ has zero Ricci tensor if and only if $H$ is $g$-harmonic and the Ricci tensor of $g$ satisfies the equation $\Ric_g = \frac14 H^2$. 
In this context, we shall say that $(g,H)$ is a {\em Bismut Ricci flat pair} ({\em BRF pair} for short) if $H$ is closed and the corresponding Bismut connection $\nabla$ is Ricci flat. 

A further motivation to consider BRF pairs relies on the fact that they provide fixed points of the {\em generalized Ricci flow}, 
a geometric flow evolving a family of Riemannian metrics $g_t$ and 2-forms $b_t$ as follows 
\[
\left\{
\begin{split}
\frac{\partial}{\partial t} g_t &= -2\,\Ric_{g_t} +\frac12 H^2_t,\\
\frac{\partial}{\partial t} b_t &= - \delta_{g_t} H_t,
\end{split}
\right.
\]
where $H_t = H + db_t$ for a given background closed 3-form $H$. 
This flow was introduced in \cite{CFMP,OSW} in the context of renormalization group flows of two-dimensional nonlinear sigma models, 
and it can be considered as a generalization of Hamilton's Ricci flow to Bismut connections with closed torsion form \cite{Str},  
and as a flow of generalized metrics on exact Courant algebroids \cite{Gar,FS, Str2}. 
Moreover, it is also related to some geometric flows in Hermitian Geometry, like the pluriclosed flow and the generalized K\"ahler Ricci flow, see e.g.~\cite{GJS,Str3,StTi0,StTi,StTi2}. 

A key example of (homogeneous) manifold with a BRF pair is provided by a semisimple compact Lie group $\G$ endowed with the bi-invariant metric $g_{\G}$ induced by $-B$, where $B$ denotes its Cartan-Killing form, 
together with the {\em standard} harmonic $3$-form $H_\G(X,Y,Z) = g_\G([X,Y],Z)$, where $X,Y,Z$ are left-invariant vector fields. 
Indeed, it is well-known that the Bismut connection on $\G$ induced by the pair $(g_{\G},H_\G)$ is flat.  
Viceversa, it has been proved that a compact simply connected Riemannian manifold carrying a flat Bismut connection with closed $3$-form is isometric to the product of compact simple Lie groups, see \cite{AF,CaSc}. 

In \cite{FS}, the authors asked whether an invariant Bismut connection with zero Ricci tensor on a homogeneous manifold should be flat, 
generalizing the well-known Alekseevsky-Kimelfeld Theorem in the Riemannian case \cite{AK}. 
In \cite{PR}, we answered this question negatively, proving the existence of invariant non-flat BRF pairs on a series of $5$-dimensional homogeneous manifolds $\M_{p,q}$ 
parametrized by a pair of positive integers with $\mathrm{gcd}(p,q)=1$. These manifolds can be represented as quotient spaces $(\SU(2)\times\SU(2))/\T^1_{\rm{diag}}$, 
for some suitably embedded torus $\T^1 \subset \SU(2)\times\SU(2)$. 

The aim of this paper is to generalize the previous result by providing a construction which allows to obtain infinite families of non-flat BRF pairs on compact homogeneous manifolds. 
These families include the manifold $\M_{1,1}$, but not $\M_{p,q}$ when $p,q\neq 1$. 
The new examples, which are constructed starting from compact symmetric spaces of inner type, admit a subcover which can be isometrically and minimally embedded into a compact semisimple Lie group in such a way that 
the harmonic $3$-form coincides with the pull-back of the standard harmonic $3$-form on the group.

Our main result can be stated as follows.  

\begin{mtheorem}
Let $\G$ be a compact, connected semisimple Lie group and let $\sigma$ be an involutive inner automorphism of $\G$. 
If $\K\subset\G$ is a compact subgroup with $(\G^\sigma)^o\subseteq \K\subseteq \G^\s$, then the homogeneous space $\M=(\G\times\G)/\K_{\mathrm{diag}}$, 
where $\K_{\mathrm{diag}} = \{(k,k)\in\G\times\G \st k\in\K\}$, is endowed with an invariant non-flat BRF pair $(\og, \oH)$.  

Moreover, if $\K = \G^\s$, there exists a $\G\times \G$-equivariant minimal embedding $\iota : \M \hookrightarrow \G\times\G$ 
so that  $(\og, \oH) = \iota^*(g_\G\oplus g_\G,H_\G\oplus H_\G)$. 
\end{mtheorem}

We remark that the pair $(\G,\K)$ is known as a {\em symmetric pair}, and that the involution $\s$ is inner if and only if $\G^\s$ has maximal rank in $\G$ (see e.g.~\cite[Ch.~IX, Thm 5.6]{Hel}). 
We also recall that every compact semisimple Lie group admits at least one involutive inner automorphism. 
The list of all symmetric pairs $(\gg,\gk)$ of compact type with $\gg$ simple and  $\mathrm{rank}\, \gg = \mathrm{rank}\, \gk$ can be deduced from \cite[Ch.~X, Table V]{Hel}. 
For the sake of completeness, we list them in Table \ref{Table1}. 

As an immediate consequence, we have the following. 
\begin{corollary*}
There exist infinitely many compact homogeneous spaces admitting an invariant non-flat BRF pair $(g,H)$. 
\end{corollary*} 
 
\begin{table}[ht]
\begin{center}
\renewcommand\arraystretch{1.4}
\begin{tabular}{|c|c|c|}
\hline
class  & $\gg$			&	$\gk$		\\ \hline \hline
$A~III$ & $\su(n)$		& $\mathfrak{s}(\gu(p)+\gu(n-p))$				\\ \hline
$BD~I$ & $\so(2n),\ n\geq 4$		& $\so(2p) + \so(2(n-p))$						\\ \hline
$BD~I$ & $\so(2n+1),\ n\geq 2$	& $\so(2p) + \so(2(n-p)+1)$						\\ \hline
$D~III$ & $\so(2n),\ n\geq 3$			& $\gu(n)$								\\ \hline
$C~I$ & $\gsp(n),\ n\geq 3$			& $\gu(n)$						\\ \hline
$C~II$ & $\gsp(p+q)$		& $\gsp(p) + \gsp(q)$				\\ \hline
$E~II$ & $\fre_6$			& $\su(6) + \su(2)$					\\ \hline
$E~III$	 & $\fre_6$			& $\so(10) + \R$			\\ \hline
\end{tabular}
\hspace{0.1cm}
\begin{tabular}{|c|c|c|}
\hline
class  & $\gg$			&	$\gk$		\\ \hline \hline
$E~V$	 & $\fre_7$			& $\su(8)$					\\ \hline
$E~VI$ & $\fre_7$			& $\so(12) + \su(2)$			\\ \hline
$E~VII$	 & $\fre_7$			& $\fre_6 + \R$				\\ \hline
$E~VIII$ & $\fre_8$			& $\so(16)$				\\ \hline
$E~IX$ & $\fre_8$			& $\fre_7+ \su(2)$			\\ \hline
$F~I$ & $\gf_4$			& $\gsp(3)+ \su(2)$				\\ \hline
$F~II$ & $\gf_4$			& $\so(9)$					\\ \hline
$G$ & $\gg_2$			& $\su(2) + su(2)$					\\ \hline
\end{tabular}
\vspace{0.2cm}
\caption{Symmetric pairs of compact type $(\gg,\gk)$ with $\gg$ simple and $\mathrm{rank}\, \gg = \mathrm{rank}\, \gk$}\label{Table1}
\end{center}
\end{table}

We note that the space $(\G\times \G)/\K_{\rm diag}$ is diffeomorphic to $\G\times (\G/\K)$ (see details in Section \ref{SecBetti}), but this diffeomorphism is not isometric 
when the latter is endowed with the product of $g_\G$ and the standard metric on the symmetric space $\G/\K$. 
We also remark that the homogeneous space $\M = (\G\times\G)/\G^\sigma_{\mathrm{diag}}$ is a $4$-{\em symmetric} space defined by means of the order four automorphism of $\G\times\G$ 
given by $(g_1,g_2) \mapsto (g_2,\sigma(g_1))$, see e.g.~\cite{Jim}. 
It is a well-known fact that for any symmetric space $\G/\G^\s$ the Cartan embedding $\phi:\G/\G^\s\to \G$ given by $\phi(a\G^\s)=\s(a)a^{-1}$ is totally geodesic. 
Moreover, the pull-back of the $3$-form $H_\G$ vanishes on $\G/\G^\s$. 
Our embedding of the $4$-symmetric space $\M$ into $\G\times \G$ is equivalent to the embedding of $\G\times (\G/\G^\sigma)$ into $\G\times \G$ via $ {\rm Id}\times \phi$,  
and it gives a generalization of the Cartan embedding for $2$-symmetric spaces, with minimality in place of the  total geodesic feature. 
Finally, we point out that $\oH$ is never zero, as otherwise $(\M,\og)$ would be Ricci flat and thus flat by \cite{AK}. 

\smallskip

The Main Theorem is proved in Section \ref{Proof}. 
We first consider the case when $\K=\G^\s$, where $\s=\tau_z$ is the conjugation by some element $z\in\G$. 
We construct an action of $\G\times \G$ on itself preserving the flat BRF pair $(g_\G\oplus g_\G,H_\G\oplus H_\G)$, and we identify a particular minimal orbit with $\M$. 
We then show that the induced metric $\og$ and $3$-form $\oH$ provide an invariant BRF pair on $\M$ whose corresponding Bismut connection $\onabla$ is not flat. 
To conclude the proof, it is then sufficient to observe that $(\G\times\G)/\K_{\mathrm{diag}}$ is a finite cover of $(\G\times\G)/\G^\s_{\mathrm{diag}}$. 
In Proposition \ref{HnotPar}, we also prove that the torsion $\oH$ is not parallel with respect to $\onabla$. 
\par \smallskip

Some remarks on the geometry of $\M$ are discussed in Section \ref{SecBetti}, where we prove the following. 
\begin{proposition}\label{PropTop}
The manifold $\M$ has finite fundamental group and $b_3(\M) = \ell$, where $\ell$ is the number of simple factors of $\G$.  
Moreover
\begin{enumerate}[1)]
\item\label{PropTop1}  if $\G$ is simple, then $b_3(\M)=1$ and $\M$ is not diffeomorphic to the product of two manifolds carrying a BRF pair with non-trivial torsion;
\item\label{PropTop2}  if $\G$ is semisimple not simple, then $\M$ is finitely covered by the product of $\ell$ factors of the form $(\G'\times\G')/\K'_{\mathrm{diag}}$, with $\G'$ simple and $(\G',\K')$ an  inner symmetric pair. 
\end{enumerate}
\end{proposition}

\medskip

{\bf Notation.} Throughout the paper, Lie groups will be denoted by capital letters and their  Lie algebras will be denoted by the respective gothic letters.
When a Lie group $\G$ acts on a manifold $\M$, the vector field associated to any $X\in\gg$ will be denoted by $\widehat X$.

\section{Proof of the Main Theorem}\label{Proof}
We keep the same notation as in Section \ref{intro}, 
and we start considering the Lie group $\Ng\coloneqq \G \times \G$ endowed with the biinvariant product metric $g \coloneqq g_{\G} \oplus g_{\G}$ together with the harmonic 3-form $H \coloneqq H_\G \oplus H_\G$.   
We denote by $\nabla$ the flat Bismut connection corresponding to $(g,H)$.   
The Lie group $\L\coloneqq \G \times \G$ acts isometrically on $(\Ng,g)$ preserving $H$ as follows 
\[
(g_1,g_2) \cdot (x_1,x_2) = \left(g_1x_1g_2^{-1},g_1x_2g_2^{-1}\right). 
\] 
The involution $\s$ of $\G$ is inner and therefore there exists $z\in \G$ so that $\s$ coincides with the conjugation $\tau_z$. 
The subgroup $\G^\s$ coincides with the centralizer $C\coloneqq C_\G(z) = \left\{a\in\G \st az=za \right\}$, and if we consider the $\L$-orbit through the point $p = (e,z)$ in $\Ng$, 
then the stabilizer $\L_p$ of $p$ is given by
\[
\L_p = C_{\mathrm{diag}} = \left\{(a,a) \st a\in C\right\},
\]
and the orbit is then
\[
\L\cdot p \cong \L/\L_p = (\G\times \G)/C_{\mathrm{diag}} = \M.
\] 
We denote by $\og$ the metric induced on $\M$ as a submanifold of $(\Ng,g)$ and by $\oH$ the pull-back to $\M$ of the $3$-form $H$. 

In order to prove that the pair $(\og,\oH)$ on $\M$ induces a Bismut connection $\onabla$ that is Ricci flat and non-flat, we split the proof into several steps which are dealt with in separate subsections.   

\subsection{Basic facts on the induced metric $\og$}
Let $\gc = \left\{u \in\gg \st {\Ad(z)u}=u \right\}$ denote the Lie algebra of $C$ and let $\gq = \left\{u \in\gg \st {\Ad(z)u}=-u \right\}$ denote its $B$-orthogonal complement in $\gg$. 
We put $k\coloneqq \dim \gc$, $q \coloneqq \dim \gq$, so that $m \coloneqq \dim \M = k+2q$ and $n \coloneqq \dim \Ng = 2k + 2q$.
\smallskip

We can consider the following decomposition of the Lie algebra $\gl = \gg \oplus \gg$ of $\L$
\[
\gl = \gc_{\mathrm{diag}} \oplus \gp \oplus \gm_1 \oplus \gm_2, 
\]
where 
\[
\gc_{\mathrm{diag}} = \{(u,u) \st u\in\gc\},~\gp = \{(u,-u) \st u\in\gc\},~\gm_1 = \{(v,0) \st v\in\gq\},~\gm_2 = \{(0,v) \st v\in\gq\}, 
\]
and there is a natural identification $\gp \oplus \gm_1 \oplus \gm_2 \cong T_p\M$ by means of
\[
V \mapsto \widehat{V}_p = \left.\frac{d}{dt}\right|_{t=0}\exp(tV)\cdot p. 
\] 
 
\smallskip

For every $V=(v_1,v_2)\in \gl$ and $(x_1,x_2)\in N$, we have 
\begin{equation}\label{Vhatgen}
\begin{split}
\widehat{V}_{(x_1,x_2)}	&= \left.\frac{d}{dt}\right|_{t=0} \left( \exp(tv_1)\, x_1 \exp(-tv_2),\exp(tv_1)\, x_2 \exp(-tv_2) \right)\\
					&= \left( dR_{x_1} (v_1) - dL_{x_1} (v_2), dR_{x_2} (v_1) - dL_{x_2} (v_2) \right) \\
					&= \left(\left(v_1^R-v_2^L\right)_{x_1}, \left(v_1^R-v_2^L\right)_{x_2}\right) \eqqcolon \left(\widehat{V}^{(1)}_{x_1},\widehat{V}^{(2)}_{x_2}\right),
\end{split}
\end{equation}
where $v^R$ and $v^L$ denote, respectively, the right-invariant and left-invariant vector field induced by $v\in\gg$ on $\G$, and $\widehat{V}^{(1)}$ and $\widehat{V}^{(2)}$ are the projections of $\widehat{V}$ 
onto the first and second factor of $T\G \times T\G$. 
Moreover, for $V=(v_1,v_2),~W=(w_1,w_2),~U=(u_1,u_2)\in \gl$, we obtain 
\begin{equation}\label{ggen}
\begin{split}
g\left(\widehat{V}_{(x_1,x_2)},\widehat{W}_{(x_1,x_2)} \right)	&= g_\G\left(\widehat{V}^{(1)}_{x_1},\widehat{W}^{(1)}_{x_1}\right) + g_\G\left(\widehat{V}^{(2)}_{x_2},\widehat{W}^{(2)}_{x_2}\right)\\
								&= -B\left(dL_{x_1}^{-1}\widehat{V}^{(1)}_{x_1},dL_{x_1}^{-1}\widehat{W}^{(1)}_{x_1}\right) -B\left(dL_{x_2}^{-1}\widehat{V}^{(2)}_{x_2},dL_{x_2}^{-1}\widehat{W}^{(2)}_{x_2}\right),
\end{split}
\end{equation}
and 
\begin{equation}\label{omegagen}
\begin{split}
H \left(\widehat{V}_{(x_1,x_2)},\widehat{W}_{(x_1,x_2)}, \widehat{U}_{(x_1,x_2)} \right)	
&= H_\G\left(\widehat{V}^{(1)}_{x_1},\widehat{W}^{(1)}_{x_1},\widehat{U}^{(1)}_{x_1}\right) + H_\G\left(\widehat{V}^{(2)}_{x_2},\widehat{W}^{(2)}_{x_2},\widehat{U}^{(2)}_{x_2}\right)\\
&= {-B}\left([dL_{x_1}^{-1}\widehat{V}^{(1)}_{x_1},dL_{x_1}^{-1}\widehat{W}^{(1)}_{x_1}], dL_{x_1}^{-1}\widehat{U}^{(1)}_{x_1}\right) \\
&\quad {-B}\left([dL_{x_2}^{-1}\widehat{V}^{(2)}_{x_2},dL_{x_2}^{-1}\widehat{W}^{(2)}_{x_2}], dL_{x_2}^{-1}\widehat{U}^{(2)}_{x_2}\right). 
\end{split}
\end{equation}
For every subspace $\mathfrak{v}\subseteq \gp \oplus \gm_1 \oplus \gm_2$, we let $\widehat{\mathfrak{v}}|_p \coloneqq \left\{\widehat{V}_p\in T_p\M \st V\in\mathfrak{v} \right\}$. 
The following lemma easily follows from \eqref{Vhatgen}.   
\begin{lemma}\label{TpMsubs} 
At the point $p=(e,z)$, we have 
\[
\widehat{\gp}|_p = \left\{ (2u, dL_{z}(2u)) \st u\in\gc \right\}, \quad \widehat{\gm}_i|_p = \left\{ \left((-1)^{i+1} v, - dL_{z} (v) \right) \st v\in\gq \right\}, \ i=1,2.
\]
\end{lemma}
\begin{proof} 
If $V=(u,-u) \in \gp$, then using $\Ad(z^{-1})|_{\gc} = \mathrm{Id}_{\gc}$, we get
\[
\begin{split}
\widehat{V}_{(e,z)}	&= \left(2u, dR_{z} (u) - dL_{z}(-u) \right) =  \left(2u, dL_{z} (\Ad(z^{-1}) u + u) \right) = (2u, dL_{z}(2u)). 
\end{split}
\]
If $V=(v,0) \in \gm_1$, then 
\[
\begin{split}
\widehat{V}_{(e,z)}	&=  \left( v, dR_{z} (v) \right) =   \left( v, dL_{z} \Ad(z^{-1}) v  \right) = \left( v, - dL_{z} (v) \right),
\end{split}
\]
since $\Ad(z^{-1})|_{\gq} = -\mathrm{Id}_{\gq}$. 
Finally, if $V=(0,v) \in \gm_2$, then
\[
\begin{split}
\widehat{V}_{(e,z)}	&= \left( - v, - dL_{z} (v) \right). 
\end{split}
\]

\end{proof}

From \eqref{ggen}, we see that the Riemannian metric $g$ at the point $p = (e,z)$ is given by 
\[
g_{(e,z)}\left(\widehat{V}_{(e,z)},\widehat{W}_{(e,z)} \right)	= -B\left(\widehat{V}^{(1)}_{e},\widehat{W}^{(1)}_{e}\right) -B\left(dL_{z}^{-1}\widehat{V}^{(2)}_{z},dL_{z}^{-1}\widehat{W}^{(2)}_{z}\right), 
\]
and we obtain the following.

\begin{lemma}
The decomposition $T_p\M = \widehat{\gp}|_p\oplus  \widehat{\gm}_1|_p\oplus  \widehat{\gm}_2|_p$ is $\overline{g}$-orthogonal. Moreover, 
\begin{enumerate}[a)]
\item\label{LemMeta} if $V=(u,-u),~W=(y,-y) \in \gp$, then $\og(\widehat{V},\widehat{W})_{(e,z)}  = -8\,B(u,y)$;
\item\label{LemMetb} if $V=(v,0),~W=(w,0) \in \gm_1$, then $\og(\widehat{V},\widehat{W} )_{(e,z)} = -2\,B(v,w)$; 
\item\label{LemMetc} if $V=(0,v),~W=(0,w) \in \gm_2$, then $\og(\widehat{V},\widehat{W} )_{(e,z)} = -2\,B(v,w)$.  
\end{enumerate}
Finally, the $g$-orthogonal complement of $T_p\M \subset T_p\Ng$ is given by
\[
T_p\M^\perp = \left\{(u,-dL_z(u)) \st u \in\gc \right\}.
\]
\end{lemma}
\begin{proof}  
We first note that $\alpha \coloneqq \Ad_\L((z,z))$ acts isometrically on $\gp\oplus\gm_1\oplus\gm_2$ with $\alpha|_{\gp}=\rm{Id},\alpha|_{\gm_i}=-\rm{Id}$, 
so that $\og\left(\widehat{\gp}|_p,\widehat{\gm}_i|_p\right)=0$, for $i=1,2$. 
As for \ref{LemMeta}), we have 
\[
g_{(e,z)}\left(\widehat{V}_{(e,z)},\widehat{W}_{(e,z)} \right)	= -B\left(2u,2y\right) -B\left(2u,2y\right) = -8\,B(u,y). 
\]
\ref{LemMetb}) is proved as follows   
\[
g_{(e,z)}\left(\widehat{V}_{(e,z)},\widehat{W}_{(e,z)} \right)	=  -B\left(v,w\right) -B\left(-v,-w\right) = -2\,B(v,w), 
\]
and \ref{LemMetc}) is proved similarly. Finally, if $V=(v,0) \in \gm_1$ and $W=(0,w) \in \gm_2$, then 
\[
g_{(e,z)}\left(\widehat{V}_{(e,z)},\widehat{W}_{(e,z)} \right)	=   -B\left(v,-w\right) -B\left(-v,-w\right) = 0. 
\]
The last claim follows from Lemma \ref{TpMsubs} and ${\rm{codim}}_\Ng \M = \dim \gc$. 
\end{proof}

\subsection{The second fundamental form of the immersion and the curvature of M}\label{SecII}   
We begin stating some general facts on the induced Bismut connection on submanifolds, giving a formula for its curvature and the relative Ricci tensor.

Let $(\Ng,g)$ be a Riemannian manifold of dimension $n = m+k$, consider a 3-form $H\in \Omega^3(\Ng)$ and let $\nabla$ be the Bismut connection associated with the pair $(g,H)$. 
The torsion tensor $T$ of $\nabla$ is related to $H$ via the identity $H(X,Y,Z) = g(T_XY,Z)$, for all $X,Y,Z\in\Gamma(T\Ng)$, and thus 
\[
\nabla_XY = \nabla^g_XY +\frac12 T_XY. 
\]
Let $\M$ be an $m$-dimensional submanifold of $\Ng$, denote by $\iota : \M \rightarrow \Ng$ the corresponding injective immersion and by $\og = \iota^*g$ the Riemannian metric induced by $g$.
If $\nu\M$ denotes the normal bundle over $\M$, then   
for any pair of vector fields $X,Y\in\Gamma(T\M)$ arbitrarily extended to $\Ng$, we have 
\[
\nabla^g_XY = \onabla^{\og}_XY + h(X,Y),  
\]
where $h\in \Gamma(S^2(T^*\M)\otimes\nu\M)$ denotes the second fundamental form of $\M$ and $\onabla^{\og}$ is the Levi Civita connection of $\og$.

The submanifold $\M$ has a natural Bismut connection $\onabla$ induced by the pair $(\og,\oH=\iota^*H)$. It is related to the Bismut connection $\nabla$ on $\Ng$ as follows
\[
\nabla_XY = \onabla_XY + h(X,Y) + \frac12(T_XY)^\perp, 
\]
where $ (T_XY)^\perp$ denotes the normal component of $T_XY$. 

We can now determine the relation between the curvature tensor $R^\nabla$ of $\nabla$ and the curvature tensor $R^{\onabla}$ of $\onabla$. 
At each point $p$ of $\M$, we consider a $g$-orthonormal basis $(\xi_1,\ldots,\xi_k)$ of the normal space $\nu_p\M$, and a standard computation shows that for $X,Y,Z,U\in T_p\M$ 
\[
\begin{split}
g(R^\nabla_{X,Y}Z,U)	&=	g(\nabla_X\nabla_YZ - \nabla_Y\nabla_XZ - \nabla_{[X,Y]}Z,U) \\
					&=	\og\left(R^{\onabla}_{X,Y}Z,U\right)  - \left(h_i(Y,Z) +\frac12 g(T_YZ,\xi_i) \right) \left( h_i(X,U)  +\frac12 g(\xi_i,T_XU) \right)\\
					&\quad + \left(h_i(X,Z) +\frac12 g(T_XZ,\xi_i) \right) \left( h_i(Y,U)  +\frac12 g(\xi_i,T_YU) \right).
\end{split}
\]
As for the relation between the Ricci tensor $\Ric^\nabla$ of $\nabla$ and the Ricci tensor $\Ric^{\onabla}$ of $\onabla$, we consider a $\og$-orthonormal basis $(e_1,\ldots,e_m)$ of $T_p\M$ and we compute 
\[
\begin{split}
\Ric^\nabla(Y,Z)	&=	g\left(R^\nabla_{e_j,Y}Z,e_j\right)\\
				&=	\Ric^{\onabla}(Y,Z)  - \left(h_i(Y,Z) +\frac12 g(T_YZ,\xi_i) \right) \left( h_i(e_j,e_j)  +\frac12 g(\xi_i,T_{e_j} e_j) \right)\\
				&\quad + \left(h_i(e_j,Z) +\frac12 g(T_{e_j}Z,\xi_i) \right) \left( h_i(Y,e_j)  +\frac12 g(\xi_i,T_Y{e_j}) \right)\\
				&=	\Ric^{\onabla}(Y,Z)  - g\left(h(Y,Z) +\frac12 T_YZ, h(e_j,e_j) \right) + g\left(h(Y,e_j),h(Z,e_j)\right)   \\
				&\quad + \frac12\, g\left(h(Z,e_j), T_Y{e_j} \right) + \frac12\, g\left(h(Y,e_j), T_{e_j}Z \right) -\frac14 g(T_Ye_j,\xi_i)g(T_Z{e_j},\xi_i). 
\end{split}
\]

Using this last expression, and recalling that $H(X,Y,Z) = g(T_XY,Z)$, we obtain the following. 
\begin{proposition}\label{RicBM}
If the Bismut connection $\nabla$ on $\Ng$ is flat, then for any $Y,Z\in T_p\M$ we have
\[
\begin{split}
\Ric^{\onabla}(Y,Z)  	&=  g\left(h(Y,Z), \mu \right) - g\left(h(Y,e_j),h(Z,e_j)\right)  + \frac14\,H(Y,e_j,\xi_i)\,H(Z,{e_j},\xi_i) \\
				&\quad +\frac12 H( Y,Z,\mu ) + \frac12\, h_i(Y,e_j)\, H(Z, e_j,\xi_i ) - \frac12\, h_i(Z,e_j)\, H(Y, e_j,\xi_i ),   
\end{split}
\]
where $(e_1,\ldots,e_m)$ is a $\og$-orthonormal basis of $T_p\M$ and $\mu\coloneqq h(e_j,e_j)$ is the mean curvature vector of $\M.$
\end{proposition}

We now go back to the proof of the Main Theorem, considering $\M =  (\G\times \G)/C_{\mathrm{diag}}$ and $\Ng=\G\times\G$. 
In the following, we compute the second fundamental form $h$ of the immersion 
$\iota:\M\hookrightarrow \Ng$. 

\smallskip

We begin with some preliminary observations. Some well-known facts are summarized in the next lemma. 
\begin{lemma}\label{prelem1}
Let $v,w\in\gg$ and denote by $v^R$ and $v^L$ the right-invariant and left-invariant vector fields induced by $v$ on $\G$, respectively. Then
\begin{enumerate}[1)]  \setlength\itemsep{0.2cm}
\item\label{prelem11} $[v^R,w^L]=0$;
\item\label{prelem12} $[v^R,w^R] = -[v,w]^R$;
\item\label{prelem13} $\nabla^{g_\G}_{v^L}w^L = \frac12 [v,w]^L$.
\end{enumerate}
\end{lemma}

As for the Levi Civita connection $\nabla^g$ at $p=(e,z)$, for every pair of vectors $V=(v_1,v_2), W=(w_1,w_2)\in\gl$, we have
\[
\begin{split}
\nabla^g_{\widehat{V}}\widehat{W}_{(e,z)} 	&= \left(\nabla^{g_\G}_{v_1^R-v_2^L}\left(w_1^R - w_2^L\right)_e, \nabla^{g_\G}_{v_1^R-v_2^L}\left(w_1^R - w_2^L\right)_z\right)\\
									&= \left(\left.\nabla^{g_\G}_{v_1^R}w_1^R + \nabla^{g_\G}_{v_2^L} w_2^L - \nabla^{g_\G}_{v_1^R}w_2^L - \nabla^{g_\G}_{v_2^L}w_1^R\right|_{e}, 
										\left.\nabla^{g_\G}_{v_1^R}w_1^R + \nabla^{g_\G}_{v_2^L} w_2^L - \nabla^{g_\G}_{v_1^R}w_2^L - \nabla^{g_\G}_{v_2^L}w_1^R\right|_{z} \right). 						
\end{split}
\]
Some useful identities are collected in the next result. 
\begin{lemma}\label{prelem2}
Let  $v,w,u\in\gg$ and let $x\in\G$, then 
\begin{enumerate}[1)]
\item $g_\G\left(\nabla^{g_\G}_{v^R}w^L,u^L\right)_x = -\frac12 B([w,u],\Ad(x^{-1})v) = g_\G\left(\nabla^{g_\G}_{w^L}v^R,u^L\right)_x$;
\item $g_\G\left(\nabla^{g_\G}_{v^R}w^R,u^L\right)_x = \frac12 B([v,w],\Ad(x) u)$.  
\end{enumerate}
\end{lemma}

\begin{proof}\ 
\begin{enumerate}[1)]
\item We have
\[
\begin{split}
2 g_\G\left(\nabla^{g_\G}_{v^R}w^L,u^L\right)	&= v^R g_\G(w^L,u^L) + w^L g_\G(u^L,v^R) - u^L g_\G(v^R,w^L)\\
									&\quad -g_\G(w^L,[v^R,u^L]) - g_\G(u^L,[w^L,v^R]) + g_\G(v^R,[u^L,w^L])\\
									&=  w^L g_\G(u^L,v^R) - u^L g_\G(v^R,w^L) + g_\G(v^R,[u^L,w^L]),
\end{split}
\]
where we used that $g_\G(w^L,u^L)$ is constant and $[v^R,u^L]= 0 = [w^L,v^R]$. Now, both $w^L$ and $u^L$ are Killing fields for the biinvariant metric $g_\G$, so we have 
\[
\begin{split}
w^L g_\G(u^L,v^R)	&= g_\G([w^L,u^L],v^R) + g_\G(u^L,[w^L,v^R]) = g_\G([w^L,u^L],v^R),\\
u^L g_\G(v^R,w^L)	&= g_\G([u^L,v^R],w^L) + g_\G(v^R,[u^L,w^L]) =  -  g_\G([w^L,u^L],v^R). 
\end{split}
\]
Therefore 
\[
2 g_\G\left(\nabla^{g_\G}_{v^R}w^L,u^L\right)_x = g_\G([w^L,u^L],v^R)_x = g_\G(dL_x[w,u],dR_xv) = g_\G([w,u],\Ad(x^{-1})v)_e. 
\]
\item With similar computations as in the previous point, we have 
\[
\begin{split}
2g_\G\left(\nabla^{g_\G}_{v^R}w^R,u^L\right) 	&= v^R g_\G(w^R,u^L) + w^R g_\G(u^L,v^R) - g_\G(u^L,[w^R,v^R])\\
									&= g_\G([v^R,w^R],u^L) = - g_\G([v,w]^R,u^L), 
\end{split}
\]
thus 
\[
2g_\G\left(\nabla^{g_\G}_{v^R}w^R,u^L\right)_x = -g_\G(dR_x[v,w],dL_x u) = - g_\G([v,w],\Ad(x)u)_e. 
\]
\end{enumerate}
\end{proof}

Let $(\kappa_1,\ldots,\kappa_k)$ be a $B$-orthogonal basis of $\gc$ such that the vectors $\xi_i \coloneqq (\kappa_i,-dL_z\kappa_i)$, $i=1,\ldots,k$, form a $g$-orthonormal basis of $\nu_p\M$. 
Then, using \ref{prelem13}) of Lemma \ref{prelem1}, Lemma \ref{prelem2}, and recalling that $g_\G$ is biinvariant and $\Ad(z)\kappa_i = \kappa_i$, we obtain
\[
\begin{split}
h_i(\widehat{V},\widehat{W})_{(e,z)}	&= g\left(\nabla^g_{\widehat{V}}\widehat{W}, \xi_i\right)_{(e,z)}\\
							&= g_\G\left(\left.\nabla^{g_\G}_{v_1^R}w_1^R + \nabla^{g_\G}_{v_2^L} w_2^L - \nabla^{g_\G}_{v_1^R}w_2^L - \nabla^{g_\G}_{v_2^L}w_1^R\right|_{e}, \kappa_i\right)\\
							&\quad +g_\G\left(\left.\nabla^{g_\G}_{v_1^R}w_1^R + \nabla^{g_\G}_{v_2^L} w_2^L - \nabla^{g_\G}_{v_1^R}w_2^L - \nabla^{g_\G}_{v_2^L}w_1^R\right|_{z},-dL_z\kappa_i \right)\\
							&= -\frac12g_\G\left([v_1,w_1] - [v_2,w_2],\kappa_i\right)  - \frac12 g_\G\left([w_2,\kappa_i],v_1\right) - \frac12 g_\G\left([v_2,\kappa_i],w_1\right)\\
							&\quad +\frac12g_\G\left([v_1,w_1]- [v_2,w_2],\kappa_i\right) + \frac12 g_\G\left([w_2,\kappa_i],\Ad(z^{-1})v_1\right) 
								+ \frac12 g_\G\left([v_2,\kappa_i],\Ad(z^{-1})w_1\right)\\
							&= - \frac12 g_\G\left([w_2,\kappa_i],v_1 -\Ad(z^{-1})v_1 \right) - \frac12 g_\G\left([v_2,\kappa_i],w_1- \Ad(z^{-1})w_1\right). 
\end{split}
\]
Using this expression, we deduce the following.
\begin{proposition}\label{hexpr}
Let $\widehat{V},\widehat{W}\in T_p\M$. Then, $h(\widehat{V},\widehat{W}) = h_i(\widehat{V},\widehat{W}) \xi_i$ may be non-zero only when $V = (v_1,0)\in\gm_1$ and $W=(0,w_2)\in\gm_2$. In such a case 
\[
h_i(\widehat{V},\widehat{W})_p	=  B([v_1,w_2],\kappa_i).   
\]
In particular, the mean curvature vector $\mu$ of $\M$ is identically zero, i.e., $\M$ is a minimal submanifold of $\Ng.$ 
\end{proposition}
\begin{proof} Let $V=(v_1,v_2),~W=(w_1,w_2) \in \gp\oplus\gm_1\oplus\gm_2$. We discuss each relevant case separately: 
\begin{enumerate}[$\bullet$]
\item if $V,W\in\gp$, then $v_2=-v_1 \in \gc$, $w_2 = -w_1 \in \gc$  and $\Ad(z^{-1})|_{\gc} = \mathrm{Id}_{\gc}$, whence 
\[
h_i(\widehat{V},\widehat{W})_{(e,z)}	= 0;
\]
\item if $V\in\gp$ and $W\in\gm_1$, then $v_2=-v_1\in\gc$, $w_1\in\gq$ and $w_2=0$, whence
\[
h_i(\widehat{V},\widehat{W})_{(e,z)}	= g_\G\left([v_1,\kappa_i],w_1\right)_e = 0, 
\]
since $[v_1,\kappa_i]\in\gc$. The analogous conclusion holds when $V\in\gp$ and $W\in\gm_2$; 
\item if $V,W\in\gm_1$, then $v_1,w_1\in\gq$ and $v_2 = 0 = w_2$, whence
\[
h_i(\widehat{V},\widehat{W})_{(e,z)}	= 0. 
\]
Similarly, one has $h_i(\widehat{V},\widehat{W})_{(e,z)}= 0$ when $V,W\in\gm_2$;
\item if $V\in\gm_1$ and $W\in\gm_2$, then $v_1\in\gq$, $v_2=0$ and $w_1=0$, $w_2\in\gq$, whence 
\[
h_i(\widehat{V},\widehat{W})_{(e,z)}	= - g_\G\left([w_2,\kappa_i],v_1  \right)_e = B\left([w_2,\kappa_i],v_1  \right) =  B([v_1,w_2],\kappa_i). 
\]
\end{enumerate}
\end{proof}

\subsection{The induced Bismut connection on M is Ricci flat and non-flat} 
We begin describing the restriction of the $2$-forms $\imath_{\xi_i}H$ to $T_p\M\times T_p\M$.  
\begin{lemma}\label{xiH}
For every $i=1,\ldots,k,$ the $2$-form $\imath_{\xi_i}H$ satisfies the following:
\begin{enumerate}[1)] \setlength\itemsep{0.1cm}
\item\label{xiH1} $\imath_{\xi_i}H\left(\widehat{Y},\cdot\right) = 0$, whenever $Y\in\gp$;
\item\label{xiH2}  $\left.\imath_{\xi_i}H \right|_{\widehat{\gm}_j|_p\times \widehat{\gm}_j|_p}= 0$, for $j=1,2$;
\item\label{xiH3} for $Y=(y_1,0)\in\gm_1$ and $Z=(0,z_2)\in\gm_2$, $\imath_{\xi_i}H\left(\widehat{Y},\widehat{Z}\right) = 2B([y_1,z_2],\kappa_i)$. 
\end{enumerate}
\end{lemma}
\begin{proof} We can compute the expression of $H$ at $p=(e,z)$ from equation \eqref{omegagen}. Using this expression together with Lemma \ref{TpMsubs}, we have what follows.  
\begin{enumerate}[1)]  \setlength\itemsep{0.1cm}
\item Let $Y=(y,-y)\in\gp$. If $U= (u,-u)\in\gp$, then 
\[
H\left(\widehat{Y},\widehat{U},\xi_i\right) = {-B}([2y,2u],\kappa_i) {-B}([2y,2u],-\kappa_i) = 0,
\]
while if $U=(u,0) \in \gm_1$, then 
\[
H\left(\widehat{Y},\widehat{U},\xi_i\right) = {-B}([2y,u],\kappa_i) {-B}([2y,-u],-\kappa_i) = 0, 
\]
since $[y,u]\in\gq$ and $\kappa_i\in\gc$ are $B$-orthogonal. The same conclusion holds when $U \in \gm_2$.
\item Let $Y=(y_1,0), Z=(z_1,0)\in\gm_1$, then 
\[
H\left(\widehat{Y},\widehat{Z},\xi_i\right) = {-B}([y_1,z_1],\kappa_i) {-B}([-y_1,-z_1],-\kappa_i) = 0. 
\]
Similarly, if  $Y=(0,y_2), Z=(0,z_2)\in\gm_2$, then 
\[
H\left(\widehat{Y},\widehat{Z},\xi_i\right) = {-B}([-y_2,-z_2],\kappa_i) {-B}([-y_2,-z_2],-\kappa_i) = 0. 
\]
\item If $Y=(y_1,0)\in\gm_1$ and $Z=(0,z_2)\in\gm_2$, then
\[
H({\widehat{Y}},\widehat{Z},\xi_i) = -B([y_1,-z_2],\kappa_i) -B([-y_1,-z_2],-\kappa_i) = 2 B([y_1,z_2],\kappa_i). 
\]
\end{enumerate}
\end{proof} 

We now focus on the Ricci tensor of the Bismut connection $\onabla$ on $\M = (\G\times \G)/C_{\mathrm{diag}}$ defined by $(\og,\oH)$. 
From Proposition \ref{RicBM}, we know that it  has the following expression at $p=(e,z)$, for every $\widehat{Y},\widehat{Z}\in T_p\M$
\begin{equation}\label{RicM}
\begin{split}
\Ric^{\onabla}(\widehat{Y},\widehat{Z})	&= - g\left(h(\widehat{Y},e_j),h(\widehat{Z},e_j)\right)  + \frac14\,H(\widehat{Y},e_j,\xi_i)\,H(\widehat{Z},{e_j},\xi_i) \\
								&\quad  + \frac12\, h_i(\widehat{Y},e_j)\, H(\widehat{Z}, e_j,\xi_i ) - \frac12\, h_i(\widehat{Z},e_j)\, H(\widehat{Y}, e_j,\xi_i ), 
\end{split}
\end{equation}
where $(e_1,\ldots,e_m)$ is a $\og$-orthonormal basis of $T_p\M$. 

\begin{proposition}\label{RicNot0}
The Bismut connection $\onabla$ on $\M$ defined by the pair $(\og,\oH)$ is Ricci flat. 
\end{proposition}
\begin{proof}
We begin observing that $\Ric^{\onabla}(\widehat{Y},\cdot)$ vanishes whenever $Y\in\gp$. 
Indeed, it follows from Proposition \ref{hexpr} that $h(\widehat{Y},\cdot) = 0$. Moreover, we have $H(\widehat{Y},{e}_j,\xi_i)=0$ by \ref{xiH1}) in Lemma \ref{xiH}. 

Let us now focus on the symmetric part of $\Ric^{\onabla}$
\[
\Ric^{\onabla,\mathrm{Sym}}(\widehat{Y},\widehat{Z}) =  - g\left(h(\widehat{Y},{e}_j),h(\widehat{Z},{e}_j)\right) + \frac14\,H({\widehat{Y}},{e}_j,\xi_i)\,H({\widehat{Z}},{{e}_j},\xi_i). 
\]
From Proposition \ref{hexpr}, we deduce that the summand 
\[
- g\left(h(\widehat{Y},{e}_j),h(\widehat{Z},{e}_j)\right) = - h_i(\widehat{Y},{e}_j) h_i(\widehat{Z},{e}_j)
\]
is not zero if and only if either $Y,Z\in\gm_1$ and $e_j\in\widehat{\gm}_2|_p$ or $Y,Z\in\gm_2$ and $e_j\in\widehat{\gm}_1|_p$ . 
We may choose the orthonormal basis $(e_1,\ldots,e_m)$ of $T_p\M$ as follows: let $(E_1,\ldots,E_q)$  be a $B$-orthogonal basis of $\gq$ such that $B(E_s,E_s)=-\tfrac12$, $s=1,\ldots,q$, then 
\begin{enumerate}[$\bullet$]
\item $e_1,\ldots,e_k$ is a $\og$-orthonormal basis of $\widehat{\gp}|_p$;
\item $e_{k+1},\ldots,e_{k+q}$ is a $\og$-orthonormal basis of $\widehat{\gm}_1|_p$ with $e_{k+s} = \widehat{(E_s,0)}$, for $s=1,\ldots,q$;
\item $e_{k+q+1},\ldots,e_{k+2q}$ is a $\og$-orthonormal basis of $\widehat{\gm}_2|_p$ with $e_{k+q+s} = \widehat{(0,E_s)}$, for $s=1,\ldots,q$. 
\end{enumerate}
If $Y=(y_1,0),Z=(z_1,0)\in\gm_1$, we then have 
\[
h_i(\widehat{Y},{e}_j) h_i(\widehat{Z},{e}_j) 	=  \sum_{s=1}^q h_i(\widehat{Y},{e}_{k+q+s}) h_i(\widehat{Z},{e}_{k+q+s})
									=  \sum_{s=1}^q B([y_1,E_s],\kappa_i) B([z_1,E_s],\kappa_i).
\]
On the other hand, if $Y=(y_1,0)\in\gm_1$, then we already know that $H({\widehat{Y}},{e}_j,\xi_i) = 0$ whenever $e_j\in\widehat{\gp}|_p$. Moreover, for $s=1,\ldots,q$, we have
\[
H({\widehat{Y}},{e}_{k+s},\xi_i) = 0,  
\]
by \ref{xiH2}) in Lemma \ref{xiH}, and 
\[
H({\widehat{Y}},{e}_{k+q+s},\xi_i) = {2 B}([y_1,E_s],\kappa_i),  
\]
by \ref{xiH3}) in Lemma \ref{xiH}. 
Therefore, for all $Y=(y_1,0),Z=(z_1,0)\in\gm_1$, we have
\[
\begin{split}
\frac14\,H({\widehat{Y}},{e}_j,\xi_i) \, H({\widehat{Z}},{e}_j,\xi_i)  	&= \frac14 \sum_{s=1}^q H({\widehat{Y}},{e}_{k+q+s},\xi_i) H({\widehat{Z}},{e}_{k+q+s},\xi_i)  \\
														&= \sum_{s=1}^q B([y_1,E_s],\kappa_i) B([z_1,E_s],\kappa_i) = h_i(\widehat{Y},{e}_j) h_i(\widehat{Z},{e}_j), 
\end{split}
\]
whence it follows that $\Ric^{\onabla,\mathrm{Sym}}(\widehat{Y},\widehat{Z}) = 0$ for all $Y,Z\in\gm_1$. 
Similarly, if $Y=(0,y_2), Z=(0,z_2)\in\gm_2$, we obtain
\[
h_i(\widehat{Y},{e}_j) h_i(\widehat{Z},{e}_j) = \sum_{s=1}^q B([y_2,E_s],\kappa_i) B([z_2,E_s],\kappa_i) = \frac14\,H({\widehat{Y}},{e}_j,\xi_i) \, H({\widehat{Z}},{e}_j,\xi_i), 
\]
since $H({\widehat{Y}},{e}_{k+s},\xi_i) = {2B}([y_2,E_s],\kappa_i)$, for $s=1,\ldots,q$, and $H({\widehat{Y}},{e}_j,\xi_i) = 0$ otherwise. 
Thus, $\Ric^{\onabla,\mathrm{Sym}}(\widehat{Y},\widehat{Z}) = 0$ for all $Y,Z\in\gm_2$. 
We still have to examine the case where $Y\in\gm_1$ and $Z\in\gm_2$. Here, we have
\[
\Ric^{\onabla,\mathrm{Sym}}(\widehat{Y},\widehat{Z}) = \frac14\, H({\widehat{Y}},{e}_j,\xi_i)\,H({\widehat{Z}},{{e}_j},\xi_i) = 0,
\]
since $H({\widehat{Y}},{e}_j,\xi_i) = 0$ for $j=1,\ldots,k+q$ and $H({\widehat{Z}},{{e}_j},\xi_i) = 0$ for $j=k+q+1,\ldots,k+2q$. 
Therefore, the symmetric part of $\Ric^{\onabla}$ vanishes. 

\smallskip

We now examine the skew-symmetric part of $\Ric^{\onabla}$ 
\[
\Ric^{\onabla,\mathrm{Skew}}(\widehat{Y},\widehat{Z}) = \frac12\, h_i(\widehat{Y},e_j)  \,H(\widehat{Z},e_j,\xi_i )  - \frac12\, h_i(\widehat{Z},e_j) \,H(\widehat{Y}, e_j,\xi_i ). 
\]
The previous discussion shows that the summands might be non-zero only when $Y=(y_1,0),Z=(z_1,0)\in\gm_1$ and when $Y=(0,y_2),Z=(0,z_2)\in\gm_2$. 
In the first case, we have 
\[
\Ric^{\onabla,\mathrm{Skew}}(\widehat{Y},\widehat{Z}) = B([y_1,E_s],\kappa_i) B([z_1,E_s],\kappa_i) - B([z_1,E_s],\kappa_i) B([y_1,E_s],\kappa_i) =0. 
\]
Similarly, $\Ric^{\onabla,\mathrm{Skew}}(\widehat{Y},\widehat{Z})=0$ also in the second case. Therefore $\Ric^{\onabla,\mathrm{Skew}}=0$. 
\end{proof}

To conclude the proof of the Main Theorem, we have to show that $\onabla$ is not flat. 
From Section \ref{SecII}, we know that the curvature tensor of $\onabla$ is given by 
\begin{equation}\label{CurvM}
\begin{split}
R^{\onabla}\left(\widehat{X},\widehat{Y},\widehat{Z},\widehat{U}\right) 	
							&=	  \left(h_i(\widehat{Y},\widehat{Z}) +\frac12 H(\widehat{Y},\widehat{Z},\xi_i) \right) \left( h_i(\widehat{X},\widehat{U})  +\frac12 H(\widehat{X},\widehat{U},\xi_i) \right)\\
							&\quad - \left(h_i(\widehat{X},\widehat{Z}) +\frac12 H(\widehat{X},\widehat{Z},\xi_i) \right) \left( h_i(\widehat{Y},\widehat{U})  +\frac12 H(\widehat{Y},\widehat{U},\xi_i) \right), 
\end{split}
\end{equation}
for all $\widehat{X},\widehat{Y},\widehat{Z},\widehat{U}\in T_p\M.$ With similar computations as in the proof of Proposition \ref{RicNot0}, we can show the following. 
\begin{proposition}
Let $X=(x_1,0),~Y=(y_1,0) \in\gm_1$ and $Z=(0,z_2),~U=(0,u_2) \in\gm_2$, then
\[
R^{\onabla}\left(\widehat{X},\widehat{Y},\widehat{Z},\widehat{U}\right)  = 4\,B([y_1,z_2],\kappa_i)\,B([x_1,u_2],\kappa_i) -4\,B([x_1,z_2],\kappa_i)\,B([y_1,u_2],\kappa_i). 
\]
\end{proposition}
\begin{proof}
Consider $Y=(y_1,0) \in\gm_1$ and $Z=(0,z_2)\in\gm_2$. 
Using Proposition \ref{hexpr} and \ref{xiH3}) in Lemma \ref{xiH}, we see that  the first factor of the first summand in the RHS of \eqref{CurvM} is 
\[
h_i(\widehat{Y},\widehat{Z}) +\frac12 H(\widehat{Y},\widehat{Z},\xi_i)  =  B([y_1,z_2],\kappa_i) +\frac12\, 2 B([y_1,z_2],\kappa_i) =  2 B([y_1,z_2],\kappa_i). 
\] 
Analogous computations hold for the second factor as well as for both factors of the second summand in \eqref{CurvM}. Our claim then follows.  
\end{proof}

As an immediate consequence, we obtain that $\onabla$ is not flat. Indeed, since $[\gq,\gq] = \gc \neq \{0\}$, we can choose $x,y\in\gq$ so that $[x,y]\in\gc$ is not zero, and we have the following. 
\begin{corollary}\label{CorRneq0}
Let $x,y\in \gq$ such that $[x,y]\in\gc$ is not zero. Then, for $X=(x,0),~Y=(y,0)\in\gm_1$ and $Z=(0,x),~U=(0,y)\in\gm_2$, we have 
\[
R^{\onabla}\left(\widehat{X},\widehat{Y},\widehat{Z},\widehat{U}\right) = -4\left(B([x,y],\kappa_i) \right)^2 = -4 \left\|[x,y] \right\|^2 \neq 0.
\]
\end{corollary}

Using the previous results, we can also show the next.   
\begin{proposition}\label{HnotPar}
The torsion $\oH$ is not parallel with respect to $\onabla$. 
\end{proposition}
\begin{proof} 
First, we recall that for every $\widehat{X},\widehat{Y},\widehat{Z},\widehat{U}\in T_p\M,$ the following identities hold (see e.g.~equations (3.20) and (3.21) in \cite{IP}):
\[
d\oH(\widehat{X},\widehat{Y},\widehat{Z},\widehat{U}) = 	\mathfrak{S}_{\widehat{X},\widehat{Y},\widehat{Z}} \left[ 
											\left(\onabla_{\widehat{X}} \oH\right)(\widehat{Y},\widehat{Z},\widehat{U})
											+2\,  \og\left(\oH(\widehat{X},\widehat{Y}),\oH(\widehat{Z},\widehat{U}) \right)\right]
											-\left(\onabla_{\widehat{U}} \oH\right)(\widehat{X},\widehat{Y},\widehat{Z}),
\] 
and
\[
\mathfrak{S}_{\widehat{X},\widehat{Y},\widehat{Z}} R^{\onabla}\left(\widehat{X},\widehat{Y},\widehat{Z},\widehat{U}\right) = 	d\oH(\widehat{X},\widehat{Y},\widehat{Z},\widehat{U}) 
																								+\left(\onabla_{\widehat{U}} \oH\right)(\widehat{X},\widehat{Y},\widehat{Z})
																-\mathfrak{S}_{\widehat{X},\widehat{Y},\widehat{Z}}\, \og\left(\oH(\widehat{X},\widehat{Y}),\oH(\widehat{Z},\widehat{U}) \right).
\]
Now, if the closed 3-form $\oH$ is $\onabla$-parallel,  
then we must have 
\[
\mathfrak{S}_{\widehat{X},\widehat{Y},\widehat{Z}} R^{\onabla}\left(\widehat{X},\widehat{Y},\widehat{Z},\widehat{U}\right) = 0, 
\]
for all $\widehat{X},\widehat{Y},\widehat{Z},\widehat{U}\in T_p\M$. As in Corollary \ref{CorRneq0}, we choose $X=(x,0),~Y=(y,0)\in\gm_1$ and $Z=(0,x),~U=(0,y)\in\gm_2$ with $[x,y]\neq 0$. Then, 
\[
R^{\onabla}\left(\widehat{X},\widehat{Y},\widehat{Z},\widehat{U}\right) = -4 \left\|[x,y] \right\|^2, 
\]
while
\[ 
R^{\onabla}\left(\widehat{Y},\widehat{Z},\widehat{X},\widehat{U}\right) = 0 = R^{\onabla}\left(\widehat{Z},\widehat{X},\widehat{Y},\widehat{U}\right), 
\]
as one can easily check using Proposition \ref{hexpr} together with Lemma \ref{xiH}. 
Therefore,  
\[
\mathfrak{S}_{\widehat{X},\widehat{Y},\widehat{Z}} R^{\onabla}\left(\widehat{X},\widehat{Y},\widehat{Z},\widehat{U}\right)  = -4 \left\|[x,y] \right\|^2 \neq 0,
\]
a contradiction. 
\end{proof}

\section{Remarks on the geometry of M}\label{SecBetti}
In this section, we give some information on the geometry of $\M = (\G\times \G)/\K_{\rm diag}$, where $\K$ is a subgroup of $\G$ as in the Main Theorem. 
In particular, we prove Proposition \ref{PropTop}. 

\smallskip

We first note that the group $\G\times \G$ acts on $\M'\coloneqq\G\times (\G/\K)$ transitively as follows 
\[
(g_1,g_2)\cdot (x,y\K) = (g_1xg_2^{-1},g_2y\K),
\]
realizing $\M'$ as the orbit through $(e,\K)$, so that $\M\cong \M'$. 
However, this diffeomorphism is not isometric if we endow $\M$ with the metric $\og$ and $\M'$ with the product of $g_\G$ and the standard metric on the symmetric space $\G/\K$.
 
\smallskip 
 
It is well-known that $\pi_1(\G)$ is finite and that any quotient space of $\G$ by a compact subgroup of maximal rank has finite fundamental group and vanishing odd Betti numbers (see e.g.~\cite{Bor}). 
Therefore, if we represent $\M$ as $\M' = \G\times (\G/\K)$, then we immediately see that $\pi_1(\M)$ is finite. Moreover, it follows by K\"unneth Theorem that 
\[
b_3(\M)=b_3(\G).
\]
The fact that $b_3(\G)=\ell$ is also well-known (see e.g.~\cite{MT}). This shows the first part of Proposition \ref{PropTop}. 

In particular, when $\G$ is simple, we have $b_3(\M)=1$, and the claim \ref{PropTop1}) of Proposition \ref{PropTop} follows 
from the fact that a manifold carrying a BRF pair with non-trivial torsion form has non-trivial third Betti number. 

Finally, we discuss the case when $\G$ is semisimple not simple, proving \ref{PropTop2}) of Proposition \ref{PropTop}.  
Let $\tilde \G$ be the universal cover of $\G$ with projection $\pi$. As $\K$ has maximal rank, it is well known that $\tilde \K\coloneqq \pi^{-1}(\K^o)$ is connected and $\tilde\G/\tilde \K$ covers $\G/\K$. 
Moreover, when we split $\tilde \G = \prod_{i=1}^\ell\G_i$ into the product of simple factors $\G_i$, then $\tilde\K = \prod_{i=1}^\ell \K_i$, where $\K_i\coloneqq\tilde\K\cap \G_i$. 
Also, $(\G_i,\K_i)$ are symmetric pairs if $(\G,\K)$ is (see e.g.~\cite{Hel}). 
This implies that the manifold $\M$ is finitely covered by a product $\prod_{i=1}^\ell\tilde\M_i$, where $\tilde\M_i$ are given by quotient spaces $(\G_i\times \G_i)/(\K_i)_{\rm{diag}}$.

\bigskip

{\bf Acknowledgements.}
The authors were supported by GNSAGA of INdAM and by the project PRIN 2017  ``Real and Complex Manifolds: Topology, Geometry and Holomorphic Dynamics''. 
The authors would like to thank Jeffrey Streets and Mario Garcia-Fern\'andez for their comments and for stimulating discussions.

\end{document}